\documentclass[12pt]{amsart}
\usepackage{amscd, amsfonts, amssymb, amsthm}
\usepackage{parskip, fullpage}
\usepackage[all]{xypic}

\newcommand\CA{{\mathcal A}}

\newcommand\CO{{\mathcal O}}

\newcommand\fg{{\mathfrak g}}

\newcommand\fp{{\mathfrak p}}

\newcommand\fl{{\mathfrak l}}

\newcommand\ft{{\mathfrak t}}
\newcommand\fz{{\mathfrak z}}

\newcommand\BBC{{\mathbb C}}

\newcommand\Ad{{\operatorname{Ad}}}

\newcommand\coexp{\operatorname{coexp}}

\newcommand\Fix{{\operatorname{Fix}}}

\newcommand\GL{\operatorname{GL}}

\newcommand\inverse{^{-1}}
\renewcommand\th{{^{\text{th}}}}

\renewcommand\r{{\operatorname{ref}}}
\newcommand\reg{{\operatorname{reg}}}

\numberwithin{equation}{section}

\theoremstyle{plain}
\newtheorem{lemma}[equation]{Lemma}
\newtheorem{theorem}[equation]{Theorem}

\newtheorem{proposition}[equation]{Proposition}


\begin{document}

\title[Invariants, arrangements, and decomposition classes]
{Invariants of reflection groups, arrangements, and
  normality of decomposition classes in Lie algebras}


\author[J.M. Douglass]{J. Matthew Douglass} \address{Department of
  Mathematics\\ University of North Texas\\ Denton TX, USA 76203}
\email{douglass@unt.edu} \urladdr{http://hilbert.math.unt.edu}

\author[G. R\"ohrle]{Gerhard R\"ohrle}
\address
{Fakult\"at f\"ur Mathematik,
Ruhr-Universit\"at Bochum,
D-44780 Bochum, Germany}
\email{gerhard.roehrle@rub.de}
\urladdr{http://www.ruhr-uni-bochum.de/ffm/Lehrstuehle/Lehrstuhl-VI/rubroehrle.html}


\maketitle
\allowdisplaybreaks

\begin{abstract}
Suppose that $W$ is a finite, unitary, reflection group acting on the complex vector space $V$ and $X$ is a subspace of $V$. Define $N$ to be the setwise stabilizer of $X$ in $W$, $Z$ to be the pointwise stabilizer, and $C=N/Z$. Then restriction defines a homomorphism from the algebra of $W$-invariant polynomial functions on $V$ to the algebra of $C$-invariant functions on $X$. In this note we consider the special case when $W$ is a Coxeter group, $V$ is the complexified reflection representation of $W$, and $X$ is in the lattice of the arrangement of $W$, and give a simple, combinatorial characterization of when the restriction mapping is surjective in terms of the exponents of $W$ and $C$. As an application of our result, in the case when $W$ is the Weyl group of a semisimple, complex, Lie algebra, we complete a calculation begun by Richardson in 1987 and obtain a simple combinatorial characterization of regular decomposition classes whose closure is a normal variety. 
\end{abstract}


\section{Introduction}

Suppose that $W$ is a finite, complex reflection group acting on the complex
vector space $V=\BBC^l$ and $X$ is a subspace of $V$. Define $N_X=\{\, w\in
W\mid w(X)=X\,\}$, the setwise stabilizer of $X$ in $W$ and $Z_X=\{\, w\in
W\mid w(x)=x\, \forall x\in X\,\}$, the pointwise stabilizer of $X$ in
$V$. Then $Z_X$ is a normal subgroup of $N_X$ and we set $C_X=N_X/Z_X$. It
is easy to see that restriction defines a homomorphism from the algebra of
$W$-invariant polynomial functions on $V$ to the algebra of $C_X$-invariant
functions on $X$, say $\rho\colon \BBC[V]^W \to \BBC[X]^{C_X}$. In this note
we consider the special case when $W$ is a Coxeter group, $V$ is the
complexified reflection representation of $W$, and $X$ is in the lattice of
the arrangement of $W$. Our main result is a simple combinatorial
characterization in terms of the exponents of $W$ and $C_X$ of when the
map $\rho$ is surjective.

As an application, our main result combined with a theorem of Richardson
\cite{richardson:normality} leads immediately to a complete, and easily
computable, classification of the regular decomposition classes in a
complex, semisimple Lie algebra whose closure is a normal variety.
 
\section{Statement of the main results}

By a \emph{hyperplane arrangement} we mean a pair $(V, \CA)$, where $\CA$ is
a finite set of hyperplanes in $V$. The arrangement of a subgroup
$C\subseteq \GL(V)$ consists of the reflecting hyperplanes of the elements
in $C$ that act on $V$ as reflections. We denote the arrangement of $C$ in
$V$ by $\CA(V,C)$. Define $C^\r$ to be the subgroup generated by the
reflections in $C$. Then obviously $\CA(V, C) = \CA(V, C^\r)$.

For general information about arrangements and reflection groups we refer
the reader to \cite{orlikterao:arrangements} and \cite{bourbaki:groupes}.

Suppose from now on that $W$ is a finite subgroup of $\GL(V)$ generated by
reflections. Unless otherwise specified, we allow the case when the
generators of $W$ are ``pseudo-reflections,'' that is, elements in $\GL(V)$
with finite order whose $1$-eigenspace is a hyperplane in $V$. For a
subspace $X$ of $V$ we have two natural hyperplane arrangements in $X$:
\begin{itemize}
\item The restricted arrangement $\CA(V,W)^X$ consisting of intersections
  $H\cap X$ for $H$ in $\CA(V,W)$ with $X\not\subseteq H$.
\item The reflection arrangement $\CA(X, C_X) =\CA(X, C_X^\r)$ consisting of
  the reflecting hyperplanes of elements in $C_X$ that act on $X$ as
  reflections.
\end{itemize}

For a free hyperplane arrangement $\CA$ we denote the multiset of exponents
of $\CA$ by $\exp(\CA)$. Terao \cite{terao:freeI} has shown that reflection
arrangements are free and that $\exp(\CA(V,W))= \coexp(W)$, where
$\coexp(W)$ denotes the multiset of coexponents of $W$.

The lattice of a hyperplane arrangement is the set of subspaces of $V$ of
the form $H_1\cap \dotsm \cap H_n$ where $\{ H_1, \dots, H_n \}$ is a subset
of $\CA$. It is known that $\CA(V, W)^X$ is free when $W$ is a Coxeter group
and $X$ is a subspace in the lattice of $\CA(V, W)$ (see
\cite{orlikterao:free}, \cite{douglass:adjoint}). Thus, in this case, we
have that (1) $\exp\left( \CA(X, C_X) \right)$, $\exp\left( \CA(V, W)^X
\right)$, and $\exp\left( \CA(V, W) \right)$ are all defined; (2)
$\exp\left( \CA(X, C_X) \right)= \exp(C_X^\r)$; and (3) $\exp\left( \CA(V,
  W) \right) =\exp(W)$.

We can now state our main result.

\begin{theorem}\label{main}
  Suppose $W$ is a finite Coxeter group, $V$ affords the reflection
  representation of $W$, and $X$ is in the lattice of the arrangement
  $\CA(V,W)$. Then the restriction mapping $\rho\colon \BBC[V]^W \to
  \BBC[X]^{C_X}$ is surjective if and only if
  \[
  \exp\left( \CA(X, C_X) \right) = \exp\left( \CA(V, W)^X \right)
  \subseteq \exp\left( \CA(V, W) \right).
  \]
\end{theorem}

To simplify the notation, in the rest of this paper we denote the
arrangements $\CA(X, C_X)$, $\CA(V, W)^X$, and $\CA(V, W)$ by $\CA(C_X)$,
$\CA^X$, and $\CA$ respectively.

In the next section, using a modification of an argument of Denef and Loeser
\cite{denefloeser:regular}, we show in Proposition \ref{surj} that if $W$ is
any complex reflection group, $X$ is in the lattice of $\CA$, $C_X= C_X^\r$,
and $\rho$ is surjective, then $\CA(C_X) = \CA^X$ and $\exp(C_X) \subseteq
\exp(W)$. It then follows that in this case $\CA^X$ is a free arrangement,
$\exp(A(C_X)) = \exp(\CA^X)$, and $\exp(C_X) \subseteq \exp(W)$. In
particular, the forward implication in the theorem holds whenever $C_X$ acts
on $X$ as a reflection group.

In \S\ref{proof} we complete the proof of Theorem \ref{main} by (1) showing
in Proposition \ref{notref} that if $W$ is a Coxeter group and $C_X$ does
not act on $X$ as a reflection group, then $\rho$ is not surjective and (2)
computing all cases in which $\exp( \CA(C_X) ) = \exp( \CA^X) \subseteq
\exp( \CA)$ for a Coxeter group $W$ and showing that $\rho$ is surjective in
these cases.

Notice that the conditions $\exp( \CA(C_X) ) = \exp( \CA^X) \subseteq \exp(
\CA)$ are not that easy to satisfy. In case $W$ is a Coxeter group of type
$A_{r-1}$, up to the action of $W$, the subspaces $X$ in the lattice of
$\CA$ are parametrized by partitions of $r$. The conditions $\exp( \CA(C_X)
) = \exp( \CA^X) \subseteq \exp( \CA)$ hold if and only if the corresponding
partition of $r$ has equal parts. For $W$ a Coxeter group of type $E_8$, up
to the action of $W$, there are forty-one possibilities for $X$, eight of
which have the property that $\exp( \CA(C_X) ) = \exp( \CA^X) \subseteq
\exp( \CA)$. All cases in which $\exp( \CA(C_X) ) = \exp( \CA^X) \subseteq
\exp( \CA)$ when $W$ is a finite, irreducible Coxeter group are listed in
Tables \ref{classical} and \ref{exceptional} in \S\ref{proof}.

In the rest of this section we explain how our main result leads to a
characterization of regular decomposition classes in a complex, semisimple
Lie algebra whose closure is a normal variety. The classification of these
decomposition classes was completed, case-by-case, for classical Weyl groups
by Richardson in 1987 \cite{richardson:normality} and extended by Broer in
1998 \cite{broer:decomposition}, again using case-by-case arguments, to
exceptional Weyl groups.

Suppose that $\fg$ is a semisimple, complex Lie algebra and $G$ is the
adjoint group of $\fg$. Motivated by a question of De Concini and Procesi
about the normality of the closure of the $G$-saturation of a Cartan
subspace for an involution of $\fg$, Richardson proved the following.

\begin{theorem}[{\cite[Theorem B]{richardson:normality}}]\label{thmB}
  Suppose that $\ft$ is a Cartan subalgebra of $\fg$, $W$ is the Weyl group
  of $(\fg, \ft)$, and $X$ is a subspace of $\ft$ with the property that
  $C_X$ acts on $X$ as a reflection group. Let $Y$ denote the closure of the
  set of elements in $\fg$ whose semisimple part is in $\Ad(G) X$. Then $Y$
  is a normal, Cohen-Macaulay variety if and only if $\rho \colon
  \BBC[\ft]^W \to \BBC[X]^{C_X}$ is surjective.
\end{theorem}

When $V=\ft$ is a Cartan subalgebra of $\fg$, a subspace $X$ of $\ft$ is in
the lattice of $\CA(\ft, W)$ if and only if there is a parabolic subalgebra
$\fp$ of $\fg$ and a Levi subalgebra $\fl$ of $\fp$ with $\ft\subseteq \fl$
so that $X=\fz$ is the center of $\fl$.

Now let $\fg_\reg$ denote the set of regular elements in $\fg$. Then
$\fg_\reg$ is the disjoint union of decomposition classes of $\fg$ (see
\cite[\S3]{borho:sheets}). A decomposition class contained in $\fg_\reg$ is
a \emph{regular decomposition class.} Suppose that $\fl$ and $\fz$ are as in
the last paragraph, $\fz_0$ is the subspace of elements in $\fz$ whose
centralizer in $\fg$ is $\fl$, and $\CO$ is the regular, nilpotent, adjoint
orbit in $\fl$. Then $\Ad(G)( \fz_0 + \CO)$ is a regular decomposition
class. Moreover, every regular decomposition class is of this form for some
$\fl$ \cite[\S3]{borho:sheets}. Therefore, combining Theorems \ref{main} and
\ref{thmB}, we obtain the following characterization of regular decomposition
classes in $\fg$ that have normal closure.

\begin{theorem}
  With the notation above, suppose that $D=\Ad(G)( \fz_0 + \CO)$ is a
  regular decomposition class in $\fg$. Then $\overline{D}$ is a normal
  variety if and only if 
  \[
  \exp( \CA (\fz, C_\fz))= \exp(\CA(\ft, W)^{\fz}) \subseteq \exp( \CA(\ft,
  W)).
  \]
\end{theorem}

Using case-by-case arguments Richardson \cite{richardson:normality}
determined all cases in which $\rho \colon \BBC [\ft]^W \to
\BBC[\fz]^{C_\fz}$ is surjective when $W$ is a Weyl group of classical
type. Broer \cite{broer:decomposition} computed almost all of the additional
cases for exceptional Weyl groups. The statement of \cite[Theorem 3.1
(e7)]{broer:decomposition} is missing one case: If $\fg$ is of type $E_7$
and $\fl$ is of type $(A_1^3)'$ (with simple roots $\alpha_2$, $\alpha_5$,
$\alpha_7$, where the labeling is as in \cite{bourbaki:groupes}), then the
restriction map $\rho$ is surjective.

\section{A preliminary result}\label{part1}

In this section we prove the following result.

\begin{proposition}\label{surj}
  Suppose $W\subseteq \GL(V)$ is a complex reflection group, $X$ is in the
  lattice of $\CA$, $C_X$ acts on $X$ as a reflection group, and the
  restriction mapping $\rho\colon \BBC[V]^W \to \BBC[X]^{C_X}$ is
  surjective. Then $\exp(C_X)\subseteq \exp(W)$ and $\CA(C_X) =
  \CA^X$. Thus, $\CA^X$ is a free arrangement and if $W$ is a Coxeter group,
  then $\exp( \CA( C_X))= \exp( \CA^X) \subseteq \exp( \CA)$.
\end{proposition}

The proof shows that if $X$ is any subspace of $V$, $C_X$ acts on $X$ as a
reflection group, and $\rho$ is surjective, then $\exp(C_X)\subseteq \exp(W)$
and $\CA(C_X) \subseteq \CA^X$. The assumption that $X$ is in the lattice of
$\CA$ is only used to conclude that $\CA^X \subseteq \CA(C_X)$.

By assumption, the restriction mapping $\rho\colon \BBC[V]^W \to
\BBC[X]^{C_X}$ is a degree-preserving, surjective homomorphism of graded
polynomial algebras and so by a result of Richardson
\cite[\S4]{richardson:normality}, we may choose algebraically independent,
homogeneous polynomials $f_1$, \dots, $f_r$ in $\BBC[V]^W$ so that
$\BBC[V]^W= \BBC[f_1, \dots, f_r]$ and $\BBC[X]^{C_X}= \BBC[\rho(f_1),
\dots, \rho(f_l)]$. Since $\exp(C_X)= \{ \deg f_1-1, \dots, \deg f_l-1\}$
and $\exp(W)= \{ \deg f_1-1, \dots, \deg f_r-1\}$, we have $\exp( C_X)
\subseteq \exp( W)$.

We next show that $\CA(C_X) \subseteq \CA^X$. Suppose $K$ is in
$\CA(C_X)$. By assumption there is a $w$ in $N_X$ so that $\Fix(w)\cap X=
K$. It is shown in \cite[Theorem~ 6.27]{orlikterao:arrangements} that
$\Fix(w)$ is in the lattice of $\CA$, say $\Fix(w)= H_1\cap \dotsm\cap H_n$,
where $H_1$, \dots, $H_n$ are in $\CA$. Then $K= H_1\cap \dots\cap H_n \cap
X$. Since $\dim K=\dim X-1$, it follows that $K=H_i\cap X$ for some $i$ and
so $K$ is in $\CA^X$.

It remains to show that $\CA^X \subseteq \CA(C_X)$. We use a variant of an
argument given by Denef and Loeser \cite{denefloeser:regular} (see also
\cite{lehrerspringer:intersection}).

Suppose that homogeneous polynomial invariants $\{f_1, \dots, f_r\}$ have
been chosen as above. Let $J$ denote the $r\times r$ matrix whose $(i,j)$
entry is $\frac {\partial f_i} {\partial x_j}$ and let $J_1$ denote the
$l\times l$ submatrix of $J$ consisting of the first $l$ rows and
columns. Then $J$ and $J_1$ are matrices of functions on $V$.  For $v$ in
$V$, let $J(v)$ and $J_1(v)$ be the matrices obtained from $J$ and $J_1$
respectively by evaluating each entry at $v$.

Then $\det J_1$ is in $\BBC[V]$ and by a result of Steinberg (see
\cite[\S6.2]{orlikterao:arrangements}) the zero set of $\rho(\det J_1) =\det
\rho(J_1)$ in $X$ is precisely $\bigcup_{K\in \CA(C_X)} K$.  Thus, to show
that $\CA^X \subseteq \CA(C_X)$ it is enough to show that if $K$ is in
$\CA^X$, then $\rho(\det J_1)$ vanishes on $K$.

Denef and Loeser have shown that if $w$ is in $W$, $v_1$ and $v_2$ are
eigenvectors for $w$ with eigenvalues $\lambda_1$ and $\lambda_2$
respectively, and $f$ in $\BBC[V]^W$ is homogeneous with degree $d$, then
$\lambda_2 D_{v_2}(f)(v_1)= \lambda_1^{1-d} D_{v_2}(f)(v_1)$, where $D_v(f)$
denotes the directional derivative of $f$ in the direction of $v$. This
proves the following lemma.

\begin{lemma}\label{zero}
  Suppose $w$ is in $W$, $x$ is in $\Fix(w)$ and $v$ in $V$ is an
  eigenvector of $w$ with eigenvalue $\lambda\ne 1$. Then $D_v(f)(x)=0$ for
  every $f$ in $\BBC[V]^W$.
\end{lemma}

Suppose $H$ is in $\CA$, $s$ is a reflection in $W$ that fixes $H$, and $v$
is orthogonal to $H$ with respect to some $W$-invariant inner product on
$V$. Since $H$ is the full $1$-eigenspace of $s$ in $V$, Lemma \ref{zero}
shows that
\begin{equation}
  \label{eq:e}
  D_v(f) \quad \text{vanishes on $H$ for every $f$ in $\BBC[V]^W$.} 
\end{equation}

By \cite[Theorem~ 6.27]{orlikterao:arrangements}, we may find $w$ in $W$
with $\Fix(w)=X$. Choose a basis $\{b_1, \dots, b_r\}$ of $V$ consisting of
eigenvectors for $w$ so that $\{b_1, \dots, b_l\}$ is a basis of $X$. Let
$\{x_1, \dots, x_r\}$ denote the dual basis of $V^*$. Since $X$ is the full
$1$-eigenspace of $w$ in $V$, Lemma \ref{zero} shows that
\begin{equation}
  \label{eq:x}
  \text{for $j>l$, $D_{b_j}(f)= \frac {\partial f} {\partial x_j}$ 
    vanishes on $X$ for every $f$ in $\BBC[V]^W$.}  
\end{equation}

Now suppose $K$ is in $\CA^X$. Say $K=H\cap X$, where $H$ is in $\CA$ with
$X\not\subseteq H$. Choose $v$ in $V$ orthogonal to $H$ with respect to a
$W$-invariant inner product. Say $v=\sum_{i=1}^r \xi_i b_i$. Define $[v]$ to
be the column vector whose $i\th$ entry is $\xi_i$ for $1\leq i\leq r$ and
$[v_1]$ to be the column vector whose $i\th$ entry is $\xi_i$ for $1\leq
i\leq l$. It follows from (\ref{eq:e}) that $J(h) \cdot [v]=0$ for every $h$
in $H$. Therefore, it follows from (\ref{eq:x}) that $J_1(k)\cdot [v_1]=0$
for every $k$ in $K$. Since $X\not\subseteq H$, we have $[v_1]\ne0$ and so
it must be the case that for $k$ in $K$, the matrix $J_1(k)$ is not
invertible.  Therefore, $\det J_1$ vanishes on $K$ and so $\rho( \det J_1)$
vanishes on $K$.  Thus, $K$ is in $\CA( C_X)$. This completes the proof of
Proposition \ref{surj}.

\section{Completion of the proof of Theorem \ref{main}}\label{proof}

In this section we complete the proof of Theorem \ref{main} and show that if
$W$ is a Coxeter group, $V$ affords the reflection representation of $W$,
and $X$ is in the lattice of $\CA$, then $\rho\colon \BBC[V]^W \to
\BBC[X]^{C_X}$ is surjective if and only if $\exp( \CA(C_X)) = \exp( \CA^X)
\subseteq \exp( \CA)$.

In the arguments below, ``degree'' means with respect to the natural grading
on $\BBC[V]$. For an integer $d$, let $\BBC[V]_d$ denote the subspace of
elements of degree $d$. For a subalgebra $R$ of $\BBC[V]$ we set $R_d= R\cap
\BBC[V]_d$. After choosing an appropriate basis of $V$ we may consider
$\BBC[X]$, $\BBC[X]^{C_X}$, and $\BBC[X]^{C_X^\r}$ as subalgebras of
$\BBC[V]$.

Also, we use the conventions that in type $A$, $A_{-1}$ and $A_0$ are to
be interpreted as the trivial group; in type $B$, $B_0$ is to be interpreted
as the trivial group and $B_1$ is to be interpreted as a component of type
$A_1$ supported on a short root; and in type $D$, $D_1$ is to be interpreted
as the trivial group and $D_2$ is to be interpreted as a component of type
$A_1 \times A_1$ supported on the two distinguished end nodes in the Coxeter
graph.

It is easy to see that if $W=W_1\times W_2$ is reducible, then Theorem
\ref{main} holds for $W$ if and only if it holds for $W_1$ and $W_2$. Thus,
we may assume that $W$ is an irreducible Coxeter group.

Fix a generating set $S$ in $W$ so that $(W,S)$ is a Coxeter system.  For a
subset $I$ of $S$ define $X_I =\cap_{s\in I} \Fix(s)$ and $W_I= \langle
I\rangle$, the subgroup of $W$ generated by $I$. Orlik and Solomon
\cite{orliksolomon:coxeter} have shown that there is a $w$ in $W$ and a
subset $I$ of $S$ so that $w(X)=X_I$, $wZ_Xw\inverse= W_I$, and
$wN_Xw\inverse = N_W(W_I)$.  Howlett \cite{howlett:normalizers} has shown
that $W_I$ has a canonical complement, $C_I$, in $N_W(W_I)$. 

We say that \emph{$C_I$ acts on $X_I$ as a Coxeter group with full rank} if
$C_I= C_I^\r$ and the Coxeter rank of $C_I$ equals the dimension of
$X_I$. For example, if $W$ is of type $E_6$ and $W_I$ is of type $A_1\times
A_2$, then $C_I=C_I^\r$ is of type $A_2$ and $\dim X_I=3$, so $C_I$ does not
act on $X_I$ as a Coxeter group with full rank. Another example is when $W$
is of type $I_2(r)$ with $r$ odd and $I$ is a one element subset of $S$. In
this case, $C_I$ is the trivial group and $X_I$ is one-dimensional.

Suppose now that the restriction mapping $\rho$ is surjective. It follows
from the next proposition that $C_X$ acts on $X$ as a Coxeter group with
full rank. In particular, we may apply Proposition \ref{surj} and conclude
that $\exp( \CA(C_X)) = \exp( \CA^X) \subseteq \exp( \CA)$. This proves the
forward implication of Theorem \ref{main}.

\begin{proposition} \label{notref} Suppose $W$ is a Coxeter group, $X$ is in
  the lattice of $\CA$, and $C_X$ does not act on $X$ as a Coxeter group
  with full rank. Then the restriction mapping $\rho\colon \BBC[V]^W \to
  \BBC[X]^{C_X}$ is not surjective.
\end{proposition}

\begin{proof}
  We may assume that $W$ is irreducible and that $X=X_I$ for some subset $I$
  of $S$. Then $W_X=W_I$, $N_X=N_W(W_I)$, and $C_X=C_I$. To show that $\rho$
  is not surjective, in each case when $C_I$ does not act on $X_I$ as a
  Coxeter group with full rank, we find an integer $d$ so that $\dim
  \BBC[V]^W_d < \dim \BBC[X_I]^{C_I}_d$. It then follows that
  $\BBC[X_I]^{C_I}_d$ is not contained in the image of $\rho$.

  If $I=\emptyset$ or $I=S$, then $C_I$ acts on $X_I$ as a Coxeter group
  with full rank. Thus, we may assume that $I$ is a non-empty, proper subset
  of $S$.

  Howlett \cite{howlett:normalizers} has computed $C_I$, $C_I^\r$, and the
  representation of $C_I$ on $X_I$ for all Coxeter groups with rank greater
  than two. When $W$ has rank two, $W$ is of type $I_2(r)$ for some $r$.  It
  is easy to see that in this case $C_I$ acts on $X_I$ as a Coxeter group
  with full rank unless $r$ is odd and $|I|=1$. Then, as noted above, $C_I$
  is the trivial group acting on the one-dimensional vector space $X_I$.
 
  The subgroup $C_I^\r$ is always a normal subgroup of $C_I$ and it turns
  out that if $C_I^\r\ne C_I$, then $C_I$ is the semidirect product of
  $C_I^\r$ with an elementary abelian $2$-group. Notice that if $w$ is any
  element in $C_I$ with order two, then $w$ acts on $X_I$ with eigenvalues
  $\pm1$, and so $w$ fixes every even degree, homogeneous, polynomial
  function on $X_I$. Therefore,
  \[
  \BBC[X_I]^{C_I}_{2n} = \BBC[X_I]^{C_I^\r}_{2n}
  \]
  for all $n$. Consequently, if either $C_I^\r$ is reducible or $C_I^\r$ is
  irreducible and the Coxeter rank of $C_I^\r$ is strictly less than the
  dimension of $X_I$, then $\dim \BBC[ X_I]^{ C_I}_2>1= \dim \BBC[V]^W_2$
  and so $\rho$ is not surjective.

  It remains to consider the cases when $C_I\ne C_I^\r$, $C_I^\r$ is
  irreducible, and the Coxeter rank of $C_I^\r$ equals $\dim X_I$.

  If $W$ is a dihedral group, then $C_I=C_I^\r$ for all $I$.

  If $W$ is of classical type and $C_I\ne C_I^\r$, then $W$ is of type $D_r$
  and $W_I$ has only components of type $A$. Suppose that this is the
  case. Then it follows from Howlett's computations
  \cite{howlett:normalizers} that whenever $C_I\ne C_I^\r$, either $C_I^\r$
  is reducible or the Coxeter rank of $C_I^\r$ is strictly less than the
  dimension of $X_I$.

  There are four cases when $C_I\ne C_I^\r$, $C_I^\r$ is irreducible, and
  the Coxeter rank of $C_I^\r$ equals $\dim X_I$: either $W$ is of type
  $E_7$ and $W_I$ is of type $A_2$, or $W$ is of type $E_8$ and $W_I$ is of
  type $A_2$, $A_1\times A_2$, or $A_4$.

  Suppose $W$ is of type $E_7$ and $W_I$ is of type $A_2$, or that $W$ is of
  type $E_8$ and $W_I$ is of type $A_1\times A_2$. We show that $\dim
  \BBC[V]^W_4< \dim \BBC[X_I]^{C_I}_4$. Fix $f_2\ne 0$ in
  $\BBC[V]^W_2$. Because the two smallest exponents of $W$ are $1,5$ and
  $1,7$, respectively, it follows that $\BBC[V]^W_4$ is one-dimensional with
  basis $\{f_2^2\}$. Since $C_I^\r$ is of type $A_5$ in both cases, we have
  $\dim \BBC[X_I]^{C_I}_4= \dim \BBC[X_I]^{C_I^\r}_4 =2$.

  Finally, suppose $W$ is of type $E_8$ and $W_I$ is of type $A_2$ or
  $A_4$. We show that $\dim \BBC[V]^W_6< \dim \BBC[X_I]^{C_I}_6$. Fix
  $f_2\ne 0$ in $\BBC[V]^W_2$. Since the two smallest exponents of $W$ are
  $1$ and $7$, it follows that $\BBC[V]^W_6$ is one-dimensional with basis
  $\{f_2^3\}$. Because $C_I^\r$ is of type $E_6$ when $W_I$ is of type $A_2$
  and that $C_I^\r$ is of type $A_4$ when $W_I$ is of type $A_4$, we have
  $\dim \BBC[X_I]^{C_I}_6= 2$ in the first case, and $\dim
  \BBC[X_I]^{C_I}_6= 3$ in the second. This completes the proof of the
  proposition.
\end{proof}

To complete the proof of Theorem \ref{main} we suppose that $\exp( C_X) =
\exp( \CA^X) \subseteq \exp( \CA)$ and show that $\rho \colon \BBC[V]^W \to
\BBC[X]^{C_X}$ is surjective. Our argument is case-by-case, using the
computation of $\exp(\CA^X)$ by Orlik and Solomon
\cite{orliksolomon:coxeter}, Howlett's results in
\cite{howlett:normalizers}, and some computer-aided computations using GAP
\cite{gap3} for six cases when $W$ is of exceptional type. For $W$ of
classical type, our argument is similar to that of Richardson
\cite{richardson:normality}, but more streamlined, especially when $W$ is of
type $D_r$, because of our assumptions on $\CA^X$.

As above, we may assume that $W$ is irreducible and that $X=X_I$ for some
proper, non-empty, subset $I$ of $S$. Then $W_X=W_I$, $N_X=N_W(W_I)$, and
$C_X=C_I$. Notice that it follows from the assumption $\exp(C_I)\subseteq
\exp(\CA)$ that $C_I^\r$ is irreducible.

Suppose first that $W$ is classical of type $A_r$, $B_r$, or $D_r$ with
$r\geq 1$, $r\geq 2$, and $r\geq 4$, respectively. Say $W_I$ has $m_i$
components of type $A_i$ and a component of type $B_j$ or $D_j$, where
$j\geq 0$. In type $A$ we set $j=-1$. Set $k= j +\sum_i (i+1)m_i$. Then $k$
is minimal so that $W_I$ may be embedded in a Coxeter group of type $A_k$,
$B_k$, or $D_k$. The group $C_I^\r$ is given as follows:
\begin{itemize}
\item $\prod_{i} A_{m_i-1} \times A_{r-k-1}$ if $W$ is of type $A_r$,
\item $\prod_{i} B_{m_i} \times B_{r-k}$ if $W$ is of type $B_r$, 
\item $\prod_{i} B_{m_i} \times B_{r-k}$ if $W$ is of type $D_r$ and
  $j\ne0$, and
\item $\prod_{\text{$i$ even}} D_{m_i} \times \prod_{\text{$i$ odd}}
  B_{m_i} \times D_{r-k}$ if $W$ is of type $D_r$ and $j=0$.
\end{itemize}

The exponents of $\CA^{X_I}$ have been computed by Orlik and Solomon
in \cite{orliksolomon:coxeter}. Set $l=\dim X_I$. Then
$\exp(\CA^{X_I})$ is given as follows:

\begin{itemize}
\item $\{1,2,3, \dots, l\}$ if $W$ is of type $A_r$,
\item $\{1,3,5, \dots, 2l-1\}$ if $W$ is of type $B_r$,
\item $\{1,3,5, \dots, 2l-1\}$ if $W$ is of type $D_r$ and $j\ne0$, and
\item $\{1,3,5 \dots, 2l-3, l-1+\sum_i m_i \}$ if $W$ is of type $D_r$
  and $j=0$.
\end{itemize}

\medskip

\subsubsection*{Type $A_r$}
Suppose $W$ is of type $A_r$. If $r-k-1>0$, then since $C_I$ is irreducible
it must be that $m_i\leq 1$ for all $i$. Then $\exp(C_I)= \{1,2, \dots,
r-\sum_i (i+1)\}$ and $\exp(\CA^{X_I})= \{1,2, \dots, r-\sum_i i\}$, and so
$r-\sum_i (i+1)= r-\sum_ii$, which is absurd.  Therefore, $r-k-1\leq
0$. Thus, $r\leq k+1$ and $W_I$ is of type $A_d^m$. In this case,
$\exp(C_I)= \{1,2, \dots ,m-1\}$, $\dim X_I= r-dm$, and $\exp(\CA^{X_I})=
\{1,2, \dots, r-dm\}$.  Therefore, $m-1=r-dm$. We conclude that $\exp(C_I)=
\exp(\CA^{X_I}) \subseteq \exp(\CA)$ if and only if $W_I$ is of type $A_d^m$,
where $r$, $d$, and $m$ are related by the equation $r+1=(d+1)m$.

Now suppose that $W_I$ is of type $A_d^m$ with $r+1=(d+1)m$. Then
identifying $W$ with the symmetric group $S_{r+1}$ acting on $\BBC^{r+1}$,
$V$ with the subspace of $\BBC^{r+1}$ consisting of all vectors whose
components sum to zero, $W_I$ with the Young subgroup $S_{d+1}^m \subseteq
S_{r+1}$, and taking the power sums as a set of fundamental polynomial
invariants for $S_{r+1}$, it is straightforward to check that $\rho$ is
surjective.

\subsubsection*{Type $B_r$}
Suppose that $W$ is of type $B_r$ with $r\geq 2$. Since $C_I$ is
irreducible, there is at most one value of $i$ with $m_i>0$. Suppose first
that there is a value of $i$ with $m_i>0$. Say $W_I$ has type $A_d^m \times
B_j$. Then we must have $r-k=0$ and so $r$, $j$, $d$, and $m$ are related by
$r=j+(d+1)m$. In this case, $C_I$ has type $B_m$ and $\dim X_I=
r-j-dm=m$. Thus $\exp(C_I)=\{\, 1,3, \dots, 2m-1\}= \exp(\CA^{X_I})$. On the
other hand, if $m_i=0$ for all $i$, then $W_I$ is of type $B_j$, $C_I$ is of
type $B_{r-j}$, $\dim X_I=r-j$, and $\exp(C_I)=\{1,3, \dots, 2(r-j)-1\}=
\exp(\CA^{X_I})$.  We conclude that $\exp(C_I)= \exp(\CA^{X_I}) \subseteq
\exp(\CA)$ if and only if $W_I$ is of type $A_d^m \times B_j$, where if
$m>0$, then $r$, $d$, $j$, and $m$ satisfy $r=j+(d+1)m$.

Now suppose that $W_I$ is of type $A_d^m \times B_j$ with $r= j+(d+1)m$ if
$m>0$. We may consider $W$ as signed permutation matrices acting on
$\BBC^r$.  Let $x_1$, \dots, $x_r$ denote the coordinate functions on
$\BBC^r$.  Then $\BBC[V]^W= \BBC[x_1, \dots, x_r]^W= \BBC[f_2, f_4, \dots,
f_{2r}]$, where $f_{2p}$ is the $p\th$ elementary symmetric function in $\{
x_1^2, \dots, x_r^2\}$. In case $m>0$, we may choose coordinate functions
$\{y_1, \dots, y_m\}$ on $X_I$ so that $C_I$ acts as signed permutations on
the coordinates and the restriction map $\BBC[V]\to \BBC[X_I]$ is given by
mapping $x_{p(d+1)+q}$ to $y_p$ for $0\leq p\leq m-1$ and $1\leq q\leq d+1$,
and $x_t$ to zero for $t>r-j=(d+1)m$.  It is then easily checked that
$\rho\colon \BBC[x_1, \dots, x_r]^W \to \BBC[y_1, \dots, y_{m}]^{C_I}$ is
surjective. In case $m=0$ we may take $C_I$ to act on the first $r-j$
components of $\BBC^r$ and so the restriction map $\BBC[V] \to \BBC[X_I]$ is
given by evaluating $x_{r-j+1}$, \dots, $x_r$ at zero.  It is now easily
checked that $\rho\colon \BBC[x_1, \dots, x_r]^W \to \BBC[x_1, \dots,
x_{r-j}]^{C_I}$ is surjective.

\subsubsection*{Type $D_r$} 
Suppose that $W$ is of type $D_r$ with $r\geq 4$. In case $j\ne0$ the
argument for type $B$ applies almost verbatim ($B_j$ is replaced by $D_j$)
and shows that $\exp(C_I)= \exp(\CA^{X_I}) \subseteq \exp(\CA)$ if and only
if $W_I$ is of type $A_d^m\times D_j$, where if $m>0$, then $r$, $d$, $j$,
and $m$ satisfy $r=j+(d+1)m$. In the case when $j=0$, the arrangement
$\CA^{X_I}$ is a Coxeter arrangement if and only if either $\sum_i m_i=0$,
in which case it is a Coxeter arrangement of type $D_l$, or $\sum_i m_i=l$,
in which case it is a Coxeter arrangement of type $B_l$. Since $\sum_i
m_i\ne0$, we must have that $\sum_i m_i=l = r-\sum_i im_i$ and $\CA^{X_I}$
is of type $B_l$. Thus, $C_I^\r$ must be of type $B_l$ and so $W_I$ must be
of type $A_d^m$, where $d$ is odd and $r=(d+1)m$. We conclude that if $j=0$,
then $\exp(C_I)= \exp(\CA^{X_I}) \subseteq \exp(\CA)$ if and only if $W_I$
is of type $A_d^m$, where $d$ is odd and $r=(d+1)m$.

Now suppose that $W_I$ is of type $A_d^m \times D_j$, where if $j,m>0$, then
$r=j+(d+1)m$, and if $j=0$, then $d$ is odd and $r=(d+1)m$. We may consider
$W$ as signed permutation matrices with determinant $1$ acting on
$\BBC^r$. Then $\BBC[V]^W= \BBC[x_1, \dots, x_r]^W= \BBC[f_2, f_4, \dots,
f_{2r-2}, g_r]$ where $f_{2p}$ is the $p\th$ elementary symmetric function
in $\{ x_1^2, \dots, x_r^2\}$ and $g_r=x_1\dotsm x_r$. The argument showing
that $\rho$ is surjective when $W$ is of type $B$ applies word for word to
show that $\rho$ is surjective in this case as well.

\medskip

In order to determine the remaining cases when $\exp(C_I) =\exp(\CA^I)
\subseteq \exp(\CA)$, we fix a root system $\Phi$ for $W$. Then
$\Phi\subseteq V^*$ and the choices of $S$ and $I$ determine a positive
system and a closed parabolic subsystem denoted by $\Phi^+$ and $\Phi_I$,
respectively. For $\alpha$ in $\Phi$, we have $\alpha|_{X_I} \ne 0$ if and
only if $\alpha\not \in \Phi_I$.

If $W_I$ is a maximal parabolic subgroup of $W$ and $\exp(C_I) =\exp(\CA^I)
\subseteq \exp(\CA)$, then $C_I$ is of type $A_1$ and acts as $-1$ on the
one-dimensional space $X_I$. By \cite[Ch.~VI \S1.1]{bourbaki:groupes}, $f_2=
\sum_{\alpha\in \Phi} \alpha^2$ is a non-zero polynomial in
$\BBC[V]^W_2$. Fix $\beta$ in $\Phi^+\setminus \Phi_I$. Then
$\{\beta|_{X_I}\}$ is a basis of $X_I^*$. If $g_2= \beta|_{X_I}^2$, then
$\BBC[X_I]^{C_I}= \BBC[ g_2]$. Since $\alpha|_{X_I}$ is a non-zero multiple
of $\beta|_{X_I}$ for $\alpha$ in $\Phi^+\setminus \Phi_I$, it follows that
$\rho(f_2)$ is a non-zero multiple of $g_2$ and so $\rho$ is surjective.

Suppose that $W$ is of type $I_2(r)$ and $|I|=1$. We have observed above
that if $r$ is odd, then $C_I$ is the trivial group, so $\exp(C_I)=\{0\}$
and $\exp(\CA^I) =\{1\}$. On the other hand, if $r$ is even, then $
\exp(C_I) =\exp(\CA^I)= \{1\}$ and $ \exp(\CA)=\{1, m-1\}$ and so $\exp(C_I)
=\exp(\CA^I) \subseteq \exp(\CA)$.

Our computations when $W$ is of classical or dihedral type are summarized in
Table \ref{classical}.

\begin{table}[h!tb]
\renewcommand{\arraystretch}{1.5}
\begin{tabular}{c|c|l}
  $W$ & $W_I$ \\
  \hline
  \hline
  $A_r$& $A_d^m$ & $r+1=(d+1)m$ \\
  $B_r$& $A_d^m B_j$ &  $m>0 \Rightarrow r=j+(d+1)m$\\ 
  $D_r$& $A_d^m D_j$ &  $[j,m>0 \Rightarrow r=j+(d+1)m]$ or 
  $[j=0 \Rightarrow m\ \text{odd}\ \wedge\ r=(d+1)m]$\\
  $I_2(r)$ & $A_1,\, \widetilde A_1$ & \text{$r$ even}
\end{tabular}
\medskip
\caption{Pairs $(W, W_I)$ with $W$ classical or dihedral, $\emptyset\ne I\ne
  S$, and $\exp(C_I)  =\exp(\CA^I) \subseteq \exp(\CA)$.} \label{classical} 
\end{table}

\medskip

Finally, suppose that $W$ is of exceptional type. The pairs $(W, W_I)$ for
which $\exp(C_I) =\exp(\CA^I) \subseteq \exp(\CA)$ are given in Table
\ref{exceptional}. The notation is as in \cite{orliksolomon:coxeter}.

\begin{table}[h!tb]
\renewcommand{\arraystretch}{1.5}
\begin{tabular}{c|cccccccccc}
  $W$ & $W_I$ \\
  \hline
  \hline
  $E_6$ & $A_2^2$ & $A_1 A_2^2$ & $A_5$ \\
  $E_7$ &  $(A_1^3)'$ & $A_1^3 A_2$ & $A_5'$ & $A_1 A_2
  A_3$ & $A_2 A_4$ & $A_1 A_5$ &  $A_6$ & $A_1 D_5$ & $D_6$ &
  $E_6$ \\ 
  $E_8$ &  $A_1 A_2 A_4$ & $A_3 A_4$ & $A_1 A_6$ & $A_7$ & $A_2 D_5$ &
  $D_7$ & $A_1 E_6$ & $E_7$ \\ 
  $F_4$ &  $A_2$ & $\widetilde A_2$ & $C_3$ & $B_3$ & $A_1 \widetilde
  A_2$ & $\widetilde A_1 A_2$  \\ 
  $G_2$ & $A_1$ & $\widetilde A_1$ \\
  $H_3$ & $A_1A_1$ & $A_2$ &$I_2(5)$ \\
  $H_4$ & $A_1A_2$ & $A_3$ &$A_1I_2(5)$ &$H_3$\\
\end{tabular}
\bigskip
\caption{Pairs $(W, W_I)$ with $W$ of exceptional type, $\emptyset\ne I\ne
  S$, and $\exp(C_I)  =\exp(\CA^I) \subseteq
  \exp(\CA)$.} \label{exceptional} 
\end{table}

We have seen above that if $W_I$ is maximal and $\exp(C_I) =\exp(\CA^I)
\subseteq \exp(\CA)$, then $\rho$ is surjective. For the remaining six
cases, $A_2^2$ in $E_6$; $(A_1^3)'$, $A_1^3 \times A_2$, and $A_5'$ in
$E_7$; and $A_2$ and $\widetilde A_2$ in $F_4$, the type of $C_I$ is given
in Table \ref{ci}.

\begin{table}[h!tb]
\renewcommand{\arraystretch}{1.5}
\begin{tabular}{c||c|ccc|cc}
  $W$ & $E_6$ & $E_7$&&& $F_4$&\\
  \hline
  $W_I$ & $A_2^2$ & $(A_1^3)'$ &  $A_1^3 A_2$ & $A_5'$ &  $A_2$ &
  $\widetilde A_2$\\ 
  \hline 
  $C_I$ & $G_2$ & $F_4$ & $G_2$&$G_2$ &$G_2$& $G_2$
\end{tabular}
\bigskip
\caption{Triples $(W, W_I, C_I)$ with $W$ of exceptional type, $\emptyset\ne
  I$,  $|I|<r-1$, and $\exp(C_I)  =\exp(\CA^I) \subseteq 
  \exp(\CA)$.} \label{ci} 
\end{table}

For these six cases, the fact that $\rho$ is surjective was checked directly
by implementing the following argument using GAP \cite{gap3} and the CHEVIE
package \cite{chevie}.

\begin{enumerate}
\item For $s$ in $S$ let $\alpha_s$ and $\omega_s$ denote the simple
  root in $V^*$ and the fundamental dominant weight in $V^*$
  determined by $s$ respectively. Then $\{\,\omega_s\mid s\notin
  I\,\}$ is a basis of $X_I^*$ and $\{\,\omega_s\mid s\notin I\,\}
  \cup \{\,\alpha_s\mid s\in I\,\}$ is a basis of $V^*$. This basis
  can be computed from the basis consisting of simple roots using the
  Cartan matrix of $W$. The restriction mapping $\BBC[V]\to \BBC[X_I]$
  is then given by evaluating $\alpha_s$ at zero for $s$ in $I$.
\item Suppose that the exponents of $W$ are $\{d_1-1, d_2-1, \dots, d_r-1\}$
  where $\{d_1-1, d_2-1, \dots, d_l-1\}$ are the exponents of $C_I$. For
  $i=1,2,\dots, l$, define $f_i=\sum_{\alpha\in \Phi^+} \alpha^{d_i}$. Even
  though $\{ f_1, \dots, f_l\}$ is not obviously algebraically independent,
  each $f_i$ is a non-zero element in $\BBC[V]^W_{d_i}$.
\item For $i=1,2, \dots, l$, express each $f_i$ as a polynomial in
  $\{\,\omega_s\mid s\notin I\,\} \cup \{\,\alpha_s\mid s\in
  I\,\}$. Then set $\alpha_s=0$ for $s$ in $I$ to get a polynomial
  $\rho(f_i)$ in $\BBC[X_I]^{C_I}_{d_i}$.
\item Compute the Jacobian determinant of $\{ \rho(f_1), \rho(f_2),
  \dots, \rho(f_l)\}$.
\end{enumerate}
It turns out that in all cases, the Jacobian determinant above is
non-zero and so it follows from \cite[Prop.~2.3]{springer:regular}
that $\BBC[X_I]^{C_I}= \BBC[ \rho(f_1), \rho(f_2), \dots, \rho(f_l)]$.
Therefore, $\rho$ is surjective. This completes the proof of Theorem
\ref{main}.



\bigskip {\bf Acknowledgments}: The authors acknowledge the financial
support of a DFG grant for the enhancement of bilateral cooperation
and the DFG-priority program SPP1388 ``Representation Theory''.  Parts
of this paper were written during a stay of the authors at the Isaac
Newton Institute for Mathematical Sciences in Cambridge during the
``Algebraic Lie Theory'' Programme in 2009, and during a visit of the
first author at the University of Bochum in 2010.

\bigskip


\bigskip

\bibliographystyle{amsalpha}

\newcommand{\etalchar}[1]{$^{#1}$}
\providecommand{\bysame}{\leavevmode\hbox to3em{\hrulefill}\thinspace}
\providecommand{\MR}{\relax\ifhmode\unskip\space\fi MR }
\providecommand{\MRhref}[2]{%
  \href{http://www.ams.org/mathscinet-getitem?mr=#1}{#2} }
\providecommand{\href}[2]{#2}


\end{document}